\documentclass[11pt,reqno]{amsart} 
\usepackage{amsmath}
\usepackage{amssymb}
\usepackage{amsthm}
\usepackage{mathrsfs}
\usepackage{graphicx}%
\usepackage{esint}
\usepackage{color}
\usepackage{enumerate}

\theoremstyle{plain} 
\numberwithin{equation}{section}
\newtheorem{theorem}{Theorem}[section]
\newtheorem{corollary}[theorem]{Corollary}

\newtheorem{lemma}[theorem]{Lemma}

\newtheorem{proposition}[theorem]{Proposition}
\theoremstyle{definition}
\newtheorem{remark}[theorem]{Remark}

\theoremstyle{definition}
\newtheorem{definition}[theorem]{Definition}

\newcommand{\supp}{\operatorname{supp}}
\newcommand{\dist}{\operatorname{dist}}

\newcommand{\good}{\operatorname{good}}
\newcommand{\I}{\operatorname{I}}

\newcommand{\C}{\mathbb{C}}

\newcommand{\abs}[1]{|#1|}
\newcommand{\Babs}[1]{\Big|#1\Big|}

\newcommand{\Norm}[2]{\|#1\|_{#2}}

\newcommand{\BNorm}[2]{\Big\|#1\Big\|_{#2}}
\newcommand{\pair}[2]{\langle #1,#2 \rangle}

\newcommand{\ave}[1]{\langle #1\rangle}

\newcommand{\Z}{\mathbb{Z}}
\newcommand{\R}{\mathbb{R}}

\begin{document}



\title[Dyadic representation of Calder\'on--Zygmund operators]{Dyadic representation and boundedness of non-homogeneous Calder\'on--Zygmund operators with mild kernel regularity}

\author{Ana Grau de la Herr\'an}
\author{Tuomas Hyt\"onen}

\address{Department of Mathematics and Statistics, P.O. Box 68, FI-00014 University of Helsinki, Finland}
\email{anagrauherran@gmail.com, tuomas.hytonen@helsinki.fi}

\thanks{Both authors were supported by the European Union through the ERC Starting Grant "Analytic-probabilistic methods for borderline singular integrals" and by the Finnish Centre of Excellence in Analysis and Dynamics Research.}

\begin{abstract}
We prove a new dyadic representation theorem with applications to the $T(1)$ and $A_2$ theorems. In particular, we obtain the non-homogeneous $T(1)$ theorem under weaker kernel regularity than the earlier approaches.
\end{abstract}

\maketitle

\tableofcontents

\section{Introduction}\label{s1}

Various results in the theory of singular integrals are known ``for all Calder\'on--Zygmund operators''. Examples that we have in mind include the $T(1)$ theorem of David--Journ\'e \cite{DJ}:
\begin{equation*}
  \Norm{T}{L^2\to L^2}\leq C\ \Leftrightarrow\ \Norm{T(1_Q)}{L^2}\leq c\abs{Q}^{1/2},\ \Norm{T^*(1_Q)}{L^2}\leq c\abs{Q}^{1/2}\ \forall Q;
\end{equation*}
its extension to non-homogeneous (non-doubling) measures by Nazarov--Treil--Volberg \cite{NTV};
and the $A_2$ theorem of the second author \cite{Hytonen:A2}:
\begin{equation*}
  \Norm{T}{L^2(w)\to L^2(w)}\leq c_T[w]_{A_2},\quad [w]_{A_2}:=\sup_Q\fint_Q w\,dx\fint_Q\frac{1}{w}dx.
\end{equation*}
However, when it comes to fine details of the definition of Calder\'on--Zygmund operators, it turns out that these theorems (seem to) require slightly different assumptions on the operator. We are particularly concerned about the minimal smoothness assumptions that one needs to impose on the kernel of the operator.

The most common definition of Calder\'on--Zygmund operators involves H\"older-continuous (in a suitable scale-invariant fashion, detailed below) kernels, with a power-type modulus of continuity $\omega(t)=t^\delta$ for $\delta\in(0,1]$, and it is in this form that both the $T(1)$ and the $A_2$ theorems first appeared. However, in many cases one can deal with more general continuity-moduli with a modified Dini-condition of the type
\begin{equation}\label{eq:logDini}
  \int_0^1\omega(t)\Big(1+\log\frac{1}{t}\Big)^\alpha\frac{dt}{t}<\infty.
\end{equation}
The usual Dini-condition corresponds to $\alpha=0$, and it is known to be enough for many classical results in the theory of singular integrals. It is only very recently that this was shown to be sufficient for the $A_2$ theorem as well, by Lacey~\cite{Lacey:elementary}, whereas the prior approaches needed $\alpha=1$ (this is explicit in \cite{HLP} and implicit in \cite{Lerner:dyadic,Lerner:simple}). The sharpest sufficient condition for the classical $T(1)$ theorem, to our knowledge, appears to be $\alpha=\frac12$, which is implicit in Figiel \cite{Figiel:1990} and explicit in Deng, Yan and Yang \cite{DYY}. However, the more recent extensions of the $T(1)$ theorem to non-homogeneous measures are only available under the H\"older-continuity assumption. (See \cite{NTV} for the full non-homogeneous $T(b)$ theorem, or \cite{Volberg:2013}, where the special case of the $T(1)$ theorem is recovered by methods that are closely related both to the $A_2$ theorem and the present paper.)

Given that the critical cancellation properties of Calder\'on--Zygmund operators are somewhat hidden in their usual definition, all results mentioned above depend crucially on (implicit or explicit) \emph{representation theorems} of Calder\'on--Zygmund operators as infinite superpositions of simpler (which often means: dyadic) model operators,
\begin{equation*}
  Tf=\sum_{i=1}^\infty \Phi_i(T,f).
\end{equation*}
In this way, the smoothness needed to bound an operator is linked with the smoothness required to obtain a convergent representation.

There are basically two kinds of representation theorems, linear and non-linear. In the linear case, each $\Phi_i$ is linear in both $T$ and $f$, or, stated otherwise, $\Phi_i(T,f)=\Psi_i(T)f$, where $\Psi_i$ is a linear transformation between suitable spaces of linear operators. Such a representation lies behind the usual proofs of the $T(1)$ theorem, as well as the original proof of the $A_2$ theorem \cite{Hytonen:A2}. In contrast to this, the more recent approaches to the $A_2$ theorem \cite{Lacey:elementary,Lerner:simple,Lerner:NYJ} (as well as some recent ramifications of the $T(1)$ theorem \cite{LaceyMart:Tb}) are based on decompositions, where $\Phi_i(T,f)$ depends non-linearly on both $T$ and $f$. Typically, such non-linear representations arise from some kind of stopping time arguments.

Although non-linear representations appear to yield stronger results, at least in questions around the $A_2$ theorem, there is still independent interest towards linear representations, which are better suited e.g., for iterative applications as in \cite{DalencOu}, or multi-parameter extensions, as in \cite{Martikainen:biparam,Ou:multiparam}. It is also of some theoretical interest whether the non-linear methods are fundamentally stronger, or whether the same results could also be recovered via linear representations.

In this paper, we prove a new (linear) dyadic representation formula with applications to both $T(1)$ and $A_2$ theorems. In the $T(1)$ direction, this leads to a non-homogeneous $T(1)$ theorem under the same mild kernel regularity assumptions that were so far only known in a homogeneous setting. As for $A_2$, while we are not able to recover the largest class of kernels amenable to non-linear methods, we come rather close to it, and much closer than any of the previously known linear arguments.

We now turn to a more detailed description of our results. Let $T$ be a Calder\'on-Zygmund operator of order $n$ on $\mathbb{R}^d$, with respect to a Borel measure $\mu$ of order $n$. That is, $T$ acts on a dense subspace of functions in $L^2(\mathbb{R}^d)$ (for the present purposes, this class should at least contain the indicators of cubes in $\mathbb{R}^d$) and has the kernel representation
\begin{equation*}
Tf(x)=\int_{\mathbb{R}^d}K(x,y)f(y)d\mu(y), \quad x\notin \supp f,
\end{equation*}
where $\mu$ satisfies the growth condition
\begin{equation*}
  \mu(B(x,r))\leq C\cdot r^n
\end{equation*}
for every $x\in\mathbb{R}^d$ and every $r>0$. Note that $\mu$ need not be a doubling measure.

Moreover, the kernel should satisfy the $n$th order standard estimates, which we assume in the following form:
\begin{equation}
|K(x,y)|\leq\frac{C}{|x-y|^n}
\end{equation}
and
\begin{equation}
|K(x,y)-K(x',y)|+|K(y,x)-K(y,x')|\leq \frac{C}{|x-y|^n}\omega\Big(\frac{|x-x'|}{|x-y|}\Big)
\end{equation}
whenever $|x-x'|\leq 1/2 |x-y|$. Here $\omega$ is a modulus of continuity: an increasing  and subadditive ($\omega(a+b)\leq\omega(a)+\omega(b)$) function with $\omega(0)=0$.

We also assume the ``$T(1)$'' conditions in the local form
\begin{equation}\label{eq:localT1}
  \begin{cases}
    \Vert T1_Q\Vert_{L^2(\mu)} &\leq C\mu(Q)^{1/2},\\
  \Vert T^*1_Q\Vert_{L^2(\mu)} &\leq C\mu(Q)^{1/2},\end{cases}
\end{equation}
for all cubes $Q\subset\R^d$, where we also regard $Q=\R^d$ as a cube in the case that $\mu(\R^d)<\infty$.

Our new representation theorem then takes the following form:

\begin{theorem}\label{thm:newRepr}
Under the above assumptions on $T$ and $\mu$, and suitable test functions $f,g\in L^2(\mu)$ (as detailed in Section~\ref{sec:finitary}), the operator $T$ admits a representation
\begin{equation*}
  \pair{Tf}{g}
  =\mathbb{E}\sum_{\substack{k\in\Z \\ k\neq 0}}\omega(2^{-\abs{k}})\Big(\pair{R_k f}{g}+\pair{Q_k f}{g}\Big)
    +\mathbb{E}\pair{\Pi_{T1}f}{g}+\mathbb{E}\pair{f}{\Pi_{T^*1}g},
\end{equation*}
where
\begin{enumerate}
 \item $\mathbb{E}$ is the expectation over a random choice of dyadic systems on $\R^d$,
 \item each $\Pi_b$, $b\in\{T1,T^*1\}$, is a dyadic paraproduct with
\begin{equation*}
  \Norm{\Pi_b}{L^2(\mu)\to L^2(\mu)}\leq C,
\end{equation*}
 \item each $R_k$ and $Q_k$ is a dyadic operator with
\begin{equation*}
  \Norm{R_k}{L^2(\mu)\to L^2(\mu)}\leq C,\qquad\Norm{Q_k}{L^2(\mu)\to L^2(\mu)}\leq C\sqrt{\abs{k}}.
\end{equation*}
\end{enumerate}
If $\mu$ is doubling, then each $R_k$ and $Q_k$ is a sum of $O(\abs{k})$ dyadic shifts of complexity $O(\abs{k})$.
\end{theorem}

Both dyadic paraproducts and dyadic shifts mentioned in the theorem have the ``usual definition'', which we shall recall below.
As indicated, the dyadic operators $R_k$ and $Q_k$ of the new representation theorem are not precisely dyadic shifts in the sense of the usual definition, although they are closely related to them. But, in a sense, these operators have a fundamental nature, and their norm bounds stated above are more efficient than what would follow from their decomposition into usual shifts and an application of known estimates.

The operators $R_k$ and $Q_k$ are related to, and inspired by, certain operators denoted by $T_m$ and $U_m$ and introduced by Figiel \cite{Figiel:1988,Figiel:1990}. These are, in fact, the first ``dyadic shifts'' in the literature, although somewhat different from the modern usage of the term.

Our main application of Theorem~\ref{thm:newRepr} is the $T(1)$ theorem, which shows that the same mild kernel regularity as in the case of the Lebesgue measure \cite{DYY} is also admissible for the general non-homogeneous (i.e., not necessarily doubling) measures $\mu$ considered above.

\begin{corollary}\label{cor:T1}
Under the above assumptions on $T$ and $\mu$, if the modulus of continuity satisfies the Dini condition \eqref{eq:logDini} with $\alpha=\frac12$, then $T$ acts boundedly on $L^2(\mu)$.
\end{corollary}

\begin{proof}[Proof assuming Theorem~\ref{thm:newRepr}]
Using the decomposition and norm estimates provided by Theorem~\ref{thm:newRepr}, it follows that
\begin{equation*}
\begin{split}
  \abs{\pair{Tf}{g}}
  &\leq\sum_{\substack{k\in\Z \\ k\neq 0}}\omega(2^{-\abs{k}})C\Big(\Norm{f}{2}\Norm{g}{2}+\sqrt{\abs{k}}\Norm{f}{2}\Norm{g}{2}\Big)
    +C\Norm{f}{2}\Norm{g}{2} \\
   &\leq C\Norm{f}{2}\Norm{g}{2}\Big(1+\sum_{k=1}^\infty\omega(2^{-k})\sqrt{k}\Big) \\
   &\leq C\Norm{f}{2}\Norm{g}{2}\Big(1+\int_0^1\omega(t)\sqrt{\log\frac{1}{t}}\frac{dt}{t}\Big)
   \leq C\Norm{f}{2}\Norm{g}{2},
\end{split}
\end{equation*}
by an easy comparison of sums and integrals.
\end{proof}

Towards the $A_2$ theorem, we obtain:

\begin{corollary}\label{cor:A2}
Let $n=d$ and $\mu$ be the Lebesgue measure on $\R^d$.
Under the above assumptions on $T$, if the modulus of continuity satisfies the Dini condition \eqref{eq:logDini} with $\alpha=2$, then $T$ satisfies the $A_2$ inequality
\begin{equation*}
  \Norm{Tf}{L^2(w)}\leq C[w]_{A_2}\Norm{f}{L^2(w)}.
\end{equation*}
\end{corollary}

While this does not quite recover the form of the $A_2$ theorem obtained by Lacey \cite{Lacey:elementary} with $\alpha=0$ (see also the simplification by Lerner \cite{Lerner:NYJ}), or even the earlier results \cite{HLP,Lerner:dyadic,Lerner:simple} using non-linear representations with $\alpha=1$, this extends the class of Calder\'on--Zygmund operators that one can handle by any of the linear representation theorems currently known.

\begin{proof}[Proof assuming Theorem~\ref{thm:newRepr}]
Since dyadic shifts of complexity $O(\abs{k})$ are bounded on $L^2(w)$ with norm $O(\abs{k}\cdot[w]_{A_2})$ (see \cite[Theorem 2.10]{HLMORSU} or \cite[Theorem 4.1]{Treil:2011}), it follows that
\begin{equation*}
  \Norm{R_k}{L^2(w)\to L^2(w)}+\Norm{Q_k}{L^2(w)\to L^2(w)}\leq Ck^2[w]_{A_2}.
\end{equation*}
Recalling that dyadic paraproducts are also bounded on $L^2(w)$ with norm $O([w]_{A_2})$ (see \cite{Beznosova}),
substituting all this into the representation formula of Theorem~\ref{thm:newRepr}, we obtain
\begin{equation*}
\begin{split}
  \Norm{T}{L^2(w)\to L^2(w)}
  &\leq C\Big(1+\sum_{k=1}^\infty\omega(2^{-k})k^2\Big)[w]_{A_2} \\
  &\leq C\Big(1+\int_0^1\omega(t)\big(\log\frac{1}{t}\big)^2\frac{dt}{t}\Big)[w]_{A_2}\leq C[w]_{A_2}.\qedhere
\end{split}
\end{equation*}
\end{proof}

We actually suspect that a direct weighted analysis of the new operators $R_k$ and $Q_k$ (instead of their reduction to known results about dyadic shifts), could lead to the better weighted bounds
\begin{equation*}
    \Norm{R_k}{L^2(w)\to L^2(w)}+\Norm{Q_k}{L^2(w)\to L^2(w)}\overset{?}{\leq} C\abs{k}[w]_{A_2},
\end{equation*}
and thus to a linear proof of the $A_2$ theorem under the Dini condition \eqref{eq:logDini} with $\alpha=1$. However, since this would still be weaker than Lacey's result \cite{Lacey:elementary} with $\alpha=0$, we have not pushed hard on this point.

For future investigations we point out that our non-homogeneous results (for power bounded measures on $\R^d$) should extend to the case of  upper doubling measures on geometrically doubling metric spaces (as in \cite{HM:nonhomogTb}) with essentially notational complications only. Somewhat less obvious may be the extension to representations appropriate for $T(b)$ (rather than just $T(1)$) theorems. Multi-linear and multi-parameter extensions should also be possible.

\subsection{Plan of the paper}
 
 We recall some preliminaries and notation in Section~\ref{s2}. Once the notation is available, we provide a more detailed statement of the main Theorem~\ref{thm:newRepr} in Section~\ref{sec:details}, including a precise formula for the various operators appearing in the decomposition. The proof of the theorem is then divided into the subsequent sections. We split the main part of the proof into \emph{identities} and \emph{estimates}, namely, writing $T$ as a sum of the dyadic pieces (Section \ref{s3}), and showing that these pieces satisfy the relevant norm bounds (Section \ref{sec:estimates}). The proof of Theorem~\ref{thm:newRepr} is completed in Section \ref{sec:shifts}, where we establish the asserted shift structure of the new operators in the homogeneous situation. The last Section \ref{sec:weak} provides additional information about the weak $(1,1)$ behaviour of the new operators, again in the homogeneous situation only.


\section{Preliminaries}\label{s2}

Starting from a fixed reference system of dyadic cubes $\mathscr{D}^0$, we shall consider new dyadic systems, obtained by translating the reference system as follows. Let $\sigma=(\sigma_j)_{j\in\mathbb{Z}}\in(\{0,1\}^d)^{\mathbb{Z}}$ and
\begin{equation*}
 I\dot{+}\sigma:=I+\sum_{j:2^{-j}<\ell(I)}2^{-j}\sigma_j.
\end{equation*}
Then $$\mathcal{D}^{\sigma}:=\{I\dot{+}\sigma: I\in\mathcal{D}^0\},$$
and it is straightforward to check that $\mathcal{D}^{\sigma}$ inherits the important nestedness property of $\mathcal{D}^0$: if $I, J\in\mathcal{D}^{\sigma}$, then $I\cap J\in\{I,J,\emptyset\}.$ When the particular $\sigma$ is unimportant, the notation $\mathcal{D}$ is sometimes used for a generic dyadic system.

The reference system could be, but need not be, the standard system
\begin{equation*}
  \mathscr{D}^0=\{2^{-k}([0,1)^d+m):k\in\Z,m\in\Z^d\};
\end{equation*}
we could equally well start from any other fixed reference system.

Within any fixed system of dyadic cubes $\mathscr{D}$, we use the following notation:
\begin{itemize}
  \item $I^{(r)}$ is the $r$th ancestor $I$, i.e, $I^{(r)}\supseteq I$ and $\ell(I^{(r)})=2^r\ell(I)$.
  \item $\operatorname{ch}(I)$ is the set of children of $I$, i.e., $J\in \operatorname{ch}(I)$ if and only if  $J^{(1)}=I$.
  \item $\mathcal{D}_k$ refers to the set of cubes $Q\in\mathcal{D}$ such that $\ell(Q)=2^{-k}$.
\end{itemize}
Moreover, we write:
\[\begin{array}{ l l}
\langle f\rangle_I=\tfrac{1}{\mu(I)}\int_If(x)d\mu(x),& \langle f,g\rangle=\int_{\mathbb{R}^d}f(x)g(x)d\mu(x)\\
E_If=1_I\langle f\rangle_I , &E_kf=\sum_{I\in\mathcal{D}_k}E_If\\
D_If=\sum_{I'\in \operatorname{ch}(I)}E_{I'}f-E_If,&  D_kf=\sum_{I\in\mathcal{D}_k}D_If.
\end{array}\]



\begin{definition}
For every $I\in\mathscr{D}$, we define the \textbf{Haar functions} as a collection of functions $\{\varphi_I^i\}_{i=1}^{2^d-1}$ such that
\begin{enumerate}
  \item $\supp\varphi_I^i\subseteq I$,
  \item $\varphi_I^i$ is constant on each $I'\in\operatorname{ch}(I)$,
  \item  $\int\varphi_I^nd\mu=0$,
  \item\label{it:HaarBound} $\Vert\varphi_I^i\Vert_{\infty}\cdot\Vert\varphi_I^i\Vert_1\leq C$ (a constant independent of $I$ and $i$),
  \item $\Vert\varphi_I^i\Vert_2\in\{0,1\}$, and
  \item $D_If=\sum_{n=1}^{2^d-1}\langle f,\varphi_I^i\rangle\varphi_I^i$.
\end{enumerate}
\end{definition}

The proof of existence of such functions can be found in \cite[Sec.~4]{Hytonen:IMRN}. If $\mu$ is doubling, the construction is well known, and in this case \eqref{it:HaarBound} can be improved to $\Norm{\varphi_I^i}{\infty}\leq C\mu(I)^{-1/2}$.

\begin{definition}


Let $k\in\mathbb{N}$ and $k\geq 2$, we say that a cube $I\in\mathcal{D}^{\sigma}$ is \textbf{$k$-good} if $\operatorname{dist}(I,\partial I^{(k)})\geq \tfrac{1}{4}\ell(I^{(k)})$.
\\

The probability of a particular cube $I\dot{+}\sigma$ being $k$-good is equal for all cubes $I\in\mathcal{D}^0$ and for all $k$ so we denote

\begin{equation*}
  \pi_{\textup{good}}:=\mathbb{P}_\sigma(I\dot{+}\sigma \ k\textup{-good})=2^{-d}.
\end{equation*}
Despite the easy numerical value of this expression, we use the notation $\pi_{\good}$ to stress the origin of this factor from goodness considerations in the relevant expressions.
\end{definition}

\begin{remark}\label{independence}
Let $Q$ be a cube, and $r,k\in\mathbb{N}$, $r,k\geq 2$, the position and $k$-goodness of $Q\dot{+}\sigma$ are independent random variables. In addition, the $r$-goodness of $Q\dot+\sigma$ is independent of the $k$-goodness of $(Q\dot+\sigma)^{(r)}$.
\end{remark}

\begin{figure}
\centering
\includegraphics[width=0.7\textwidth]{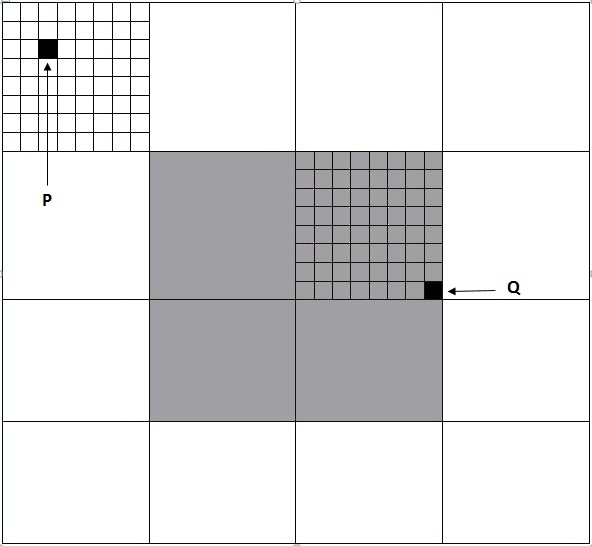}
\caption{$k$-goodness with $k=5$. The small cubes in the picture (which have $2^5$ times smaller side-length than the big cube) are $5$-good whenever they belong to the grey region (like $Q$) and $5$-bad whenever they belong to the rest of the big cube (like $P$). }
\end{figure}

%
%

The position of the translated interval by definition, depends only on $\sigma_j$ for $2^{-j}<\ell(R)$. On the other hand the $k$-goodness depends on its relative position with respect of $R^{(k)}$. Since the same translation components  $\sigma_j$ for $2^{-j}<\ell(R)$ appear in both $R\dot{+}\sigma$ and $(R\dot{+}\sigma)^{(k)}$ so that the $k$-goodness depends only on $\sigma_j$ for $\ell(R)\leq 2^{-j}<\ell(R^{(k)})$ which makes the position and $k$-goodness independent random variables.

\begin{definition}
Let $f\in L^1_{loc}(\mathbb{R}^d)$, $r\geq 2$ and $H\in\mathcal{D}$, then we define
\begin{equation}
D_H^{(r)}f:=\sum_{I:I^{(r)}=H}D_If \quad , \quad P_H^{(r)}f:=\sum_{j=0}^{r-1}D_H^{(j)}f
\end{equation}

and

\begin{equation}
D_H^{(r, good)}f:=\sum_{\substack {I:I^{(r)}=H\\ I \ r\textup{-good}}}D_If , \qquad P_H^{(r,good)}f:=\sum_{j=2}^{r-1}D_H^{(j,good)}f.
\end{equation}
\end{definition}

Due to orthogonality, it is a standard computation to verify that, for any $k\geq 1$,
\begin{equation*}
\Big(\sum_{K}\Vert D_K^{(k)}f\Vert_2^2\Big)^{1/2}\leq \Vert f\Vert_2.
\end{equation*}
%
%
%


\subsection{A finitary set-up}\label{sec:finitary}

The core of the proof of Theorem~\ref{thm:newRepr} is based on somewhat elaborate manipulations of the random martingale difference expansions of $f$ and $g$ in the duality pairing $\pair{Tf}{g}$. In order to minimise the need of tiresome justifications of the rearrangements of sums and expectations, we choose the following set-up, similar to \cite{Hytonen:2wH} or  \cite{Volberg:2013}, to carry out these manipulations:

Let $f,g\in L^2(\mu)$ be both supported on some big cube $K_1$, and constant on the dyadic subcubes of $K_1$ of side-length $2^{-N}\ell(K_1)$. Clearly, such functions are dense in $L^2(\mu)$ when both $K_1$ and $N$ are allowed to vary. Let $K_0$ be the cube with $\ell(K_0)=2\ell(K_1)$, positioned so that $K_1$ is the ``upper right'' quadrant of $K_0$.

Let us then take as the reference dyadic system some $\mathscr{D}^0$ that contains $K_0$. (This determines the choice of dyadic cubes smaller than $K_0$ uniquely, but we can make an arbitrary choice of the dyadic ancestors of $K_0$ in the system $\mathscr{D}^0$.) And then we consider a modified version of the shifted systems, with
\begin{equation*}
  K\dot+\sigma:=K+\sum_{j: 2^{-N}\ell(K_1)\leq 2^{-j}<\min(\ell(K),\ell(K_1))}2^{-j}\sigma_j.
\end{equation*}
That is, we only randomise on length scales larger than the minimal length scale $2^{-N}\ell(K_1)$, where the functions $f$ and $g$ are constant. Thus $\mathscr{D}^\sigma$ and $\mathscr{D}^0$ have the same small dyadic cubes, and since $f$ and $g$ are constant on these cubes, the corresponding martingale differences vanish.

Similarly, we only randomise on length scales smaller than the support of the functions. In particular, no matter the value of $\sigma$, we find that $K_0\dot+\sigma\supset K_1$ contains the supports of $f$ and $g$. Thus, irrespective of the value of $\sigma$, we have
\begin{equation*}
  f=E_{K_0\dot+\sigma}f+\sum_{\substack{K\in\mathscr{D}^\sigma:\,K\subseteq K_0\dot+\sigma \\ \ell(K)>2^{-n}\ell(K_1)}}D_K f,
\end{equation*}
with a similar expansion for $g$. Now, there is only a fixed finite number of terms in this sum, independently of $\sigma$. And, in effect, there are only finitely many relevant values of $\sigma$, since only the $N$ coordinates $\sigma_j$ with $2^{-N}\ell(K_1)\leq 2^{-j}<\ell(K_1)$ have an impact on the definition of $K\dot+\sigma$. Thus, all our sums and expectations range over fixed finite ranges only, so that the rearrangement of the summation order is never an issue. Since the terms $E_{K_0\dot+\sigma}f$ and $E_{K_0\dot+\sigma}g$ are trivial to handle by the $T(1)$ conditions \eqref{eq:localT1}, we may also assume that these terms (equivalently, the integrals $\int f d\mu$ and $\int g d\mu$) vanish, so that both $f$ and $g$ are expanded in terms of the martingale differences alone.

With this reduction stated, however, we will no indicate it explicitly, but simply write $\sum_{K\in\mathscr{D}}$, or just $\sum_K$, for the relevant sum above.

A cautious reader might notice that our fiddling with the definition of ``$K\dot+\sigma$'' will have some impact on the properties of the notion of goodness discussed above. However, the changes are essentially immaterial, since goodness is really only relevant to us when both the smaller and the bigger cube appear in our martingale difference expansion, and for such cubes, our randomisation is essentially unchanged. We will leave the detailed verification of this ``no harm done'' assertion to the interested reader; see also the discussion of ``good'' and ``really good'' cubes in \cite{Volberg:2013}.

\section{The dyadic representation theorem: detailed statement}\label{sec:details}

With the notation defined, we can provide additional detail to Theorem~\ref{thm:newRepr} as follows:

\begin{theorem}\label{thm:details}
In the Representation Theorem~\ref{thm:newRepr}, the various operators have the following form:
\begin{equation*}
  Q_k=\sum_{K\in\mathscr{D}}A_K^{(k)},\qquad
  R_k=\sum_{K\in\mathscr{D}}B_K^{(k)},
\end{equation*}
where the operators $A_K^{(k)}$ and $B_K^{(k)}$ satisfy the orthogonality relations
\begin{equation*}
\begin{split}
  A_K^{(k)} &=P_K^{(k+1)}A_K^{(k)}D_K^{(k+r)},\\
  A_K^{(-k)} &=D_K^{(k+r)}A_K^{(-k)}P_K^{(k+1)},\\
  B_K^{(k)} &=(P_K^{(k+r+1)}-P_K^{(k)})B_K^{(k)}D_K^{(k+r)},\\
  B_K^{(-k)} &=D_K^{(k+r)}B_K^{(-k)}(P_K^{(k+r)}-P_K^{(k)}),\qquad k\geq 1,
\end{split}
\end{equation*}
and their kernels $a_K^{(k)}(x,y)$ and $b^{(k)}_K(x,y)$ satisfy the bounds
\begin{equation*}
\begin{split}
    \abs{a_K^{(k)}(x,y)}
    &\leq C\frac{1_K(x)1_K(y)}{\ell(K)^n}+C\sum_{H:H^{(\abs{k})}=K}\frac{1_H(x)1_H(y)}{\mu(H)},\quad\abs{k}\geq 1, \\
    \abs{b_K^{(k)}(x,y)}
    &\leq C\frac{1_K(x)1_K(y)}{\ell(K)^n},\qquad  \abs{k}\geq 2.
\end{split}
\end{equation*}
Moreover, the paraproducts have the form
\begin{equation*}
  \Pi_b f
  =\sum_K D_K^{(r,\good)}b\Big(\ave{f}_K-E_{-\infty}f\Big),
\end{equation*}
where
\begin{equation*}
  E_{-\infty}f:=\begin{cases} \mu(\R^d)^{-1}\int_{\R^d}f d\mu, & \text{if }\mu(\R^d)<\infty, \\ 0, & \text{otherwise}.\end{cases}
\end{equation*}

If $\mu$ is doubling, the following additional statements hold:
\begin{itemize}
  \item The pointwise bound for $b_K^{(k)}(x,y)$ is also valid for $\abs{k}=1$.
  \item For each $U_k\in\{R_k,Q_k,R_k^*,Q_k^*\}$, we have $\Norm{U_k}{L^1(\mu)\to L^{1,\infty}(\mu)}\leq C\abs{k}$.
\end{itemize}
\end{theorem}

The proof of Theorems~\ref{thm:newRepr} and \ref{thm:details} naturally splits into two parts: the algebraic identities that give the desired decomposition, and the estimates for the terms of this expansion. We deal with each of these tasks in turn in the following two sections.

\section{Dyadic representation: identities}\label{s3}

The following proposition is introduces goodness into the basic martingale difference expansion:

\begin{proposition}\label{dyarepthm}
Let $T$ be a Calder\'on-Zygmund operator, $f,g\in L^2(\mu)$ as in the finitary set-up of Section~\ref{sec:finitary}, and $r\in\mathbb{N}$ with $r\geq 2$. Then $T$ has the following expansion
\begin{equation*}
\begin{split}
\langle  Tf,g\rangle &=\frac{1}{\pi_{\textup{good}}}\mathbb{E}_{\sigma}\sum_{I,J\in\mathcal{D}^{\sigma}}1_{r\textup{-good}}(
    \operatorname{smaller}\{I,J\})\cdot \langle TD_If,D_Jg\rangle \\
    &= \frac{1}{\pi_{\textup{good}}}\mathbb{E}_{\sigma} \Big(
    \sum_{i,j\in\Z}\langle T(\chi_{ij}D_if),\psi_{ij}D_jg\rangle \Big)
\end{split}
\end{equation*}
where
\begin{equation*}
  \operatorname{smaller}\{I,J\}:=\begin{cases}
  I, & \mbox{if $\ell(I)\leq\ell(J)$},\\
  J, & \mbox{ if $\ell (I)>\ell(J)$},\end{cases}\qquad
  \phi_j=\sum_{\substack{I\in\mathcal{D}_j \\ I \ r\textup{-good}}}1_I,
\end{equation*}
and
\begin{equation*}
  \chi_{ij}:=
  \begin{cases}
   \phi_i, & \mbox{if $i\geq j$},\\
    1 & \mbox{ if $i<j$},\end{cases}\qquad
  \psi_{ij}:=
  \begin{cases}
  1, & \mbox{if $i\geq j$},\\
   \phi_j, & \mbox{ if $i<j$}.\end{cases}
\end{equation*}
\end{proposition}

\begin{proof}
This result is similar to Proposition 3.5 of \cite{Hytonen:repre} which we are going to replicate here with our definition of goodness.

Recall that 
\begin{equation*}
f=\sum_{I\in\mathcal{D}^0}D_{I\dot{+}\sigma}f
\end{equation*}
for any $\sigma\in(\{0,1\}^d)^{\mathbb{Z}}$; and we can also take the expectation $\mathbb{E}_{\sigma}$ of both sides of this identity.


We make use of the above random $D_{I\dot{+}\sigma}$ expansion of f, multiply and divide by
\begin{equation*}
\pi_{\textup{good}}=\mathbb{E}_{\sigma} 1_{r\textup{-good}}(I\dot{+}\sigma)
\end{equation*}
and use the independence from Remark \ref{independence} to get:
\begin{align*}
\langle Tf,g\rangle &=\mathbb{E}_{\sigma}\sum_I\langle TD_{I\dot{+}{\sigma}}f,g\rangle\\
&=\tfrac{1}{\pi_{\textup{good}}}\sum_I\mathbb{E}_{\sigma}1_{r\textup{-good}}(I\dot{+}{\sigma})\mathbb{E}_{\sigma}\langle TD_{I\dot{+}{\sigma}}f,g\rangle\\
&=\tfrac{1}{\pi_{\textup{good}}}\mathbb{E}_{\sigma}\sum_I1_{r\textup{-good}}(I\dot{+}{\sigma})\langle T D_{I\dot{+}{\sigma}}f,g\rangle\\
&=\tfrac{1}{\pi_{\textup{good}}}\mathbb{E}_{\sigma}\sum_{I,J}1_{r\textup{-good}}(I\dot{+}{\sigma})\langle TD_{I\dot{+}{\sigma}}f,D_{J\dot{+}{\sigma}}g\rangle.
\end{align*}

On the other hand, using independence again in half of this double sum, we have
\begin{align*}
&\tfrac{1}{\pi_{\textup{good}}}\sum_{\ell(I)>\ell(J)}\mathbb{E}_{\sigma}1_{r\textup{-good}}(I\dot{+}{\sigma})\langle TD_{I\dot{+}{\sigma}}f,D_{J\dot{+}{\sigma}}g\rangle\\
&=\tfrac{1}{\pi_{\textup{good}}}\sum_{\ell(I)>\ell(J)}\mathbb{E}_{\sigma}1_{r\textup{-good}}(I\dot{+}{\sigma})\mathbb{E}_{\sigma}\langle TD_{I\dot{+}{\sigma}}f,D_{J\dot{+}{\sigma}}g\rangle\\
&=\mathbb{E}_{\sigma}\sum_{\ell(I)>\ell(J)}\langle TD_{I\dot{+}{\sigma}}f,D_{J\dot{+}{\sigma}}g\rangle,
\end{align*}
and hence
\begin{align*}
\langle Tf,g\rangle=\tfrac{1}{\pi_{\textup{good}}}&\mathbb{E}_{\sigma}\sum_{\ell(I)\leq\ell(J)}1_{r\textup{-good}}(I\dot{+}{\sigma})\langle TD_{I\dot{+}{\sigma}}f,D_{J\dot{+}{\sigma}}g\rangle\\
+&\mathbb{E}_{\sigma}\sum_{\ell(I)>\ell(J)}\langle TD_{I\dot{+}{\sigma}}f,D_{J\dot{+}{\sigma}}g\rangle.
\end{align*}

Comparison with the basic identity
\begin{equation*}
\langle Tf,g\rangle=\mathbb{E}_{\sigma}\sum_{I,J}\langle TD_{I\dot{+}{\sigma}}f,D_{J\dot{+}{\sigma}}g\rangle
\end{equation*}
shows that 
\begin{align*}
&\mathbb{E}_{\sigma}\sum_{\ell(I)\leq\ell(J)}\langle TD_{I\dot{+}{\sigma}}f,D_{J\dot{+}{\sigma}}g\rangle\\
&=\tfrac{1}{\pi_{\textup{good}}}\mathbb{E}_{\sigma}\sum_{\ell(I)\leq\ell(J)}1_{r\textup{-good}}(I\dot{+}{\sigma})\langle TD_{I\dot{+}{\sigma}}f,D_{J\dot{+}{\sigma}}g\rangle.
\end{align*}

Symmetrically, we also have
\begin{align*}
&\mathbb{E}_{\sigma}\sum_{\ell(I)>\ell(J)}\langle TD_{I\dot{+}{\sigma}}f,D_{J\dot{+}{\sigma}}g\rangle\\
&=\tfrac{1}{\pi_{\textup{good}}}\mathbb{E}_{\sigma}\sum_{\ell(I)>\ell(J)}1_{r\textup{-good}}(J\dot{+}{\sigma})\langle TD_{I\dot{+}{\sigma}}f,D_{J\dot{+}{\sigma}}g\rangle,
\end{align*}
and this completes the proof of the first asserted identity. The second one is a simple restatement, as seen by the computation 
\begin{align*}
\sum_{I,J} &1_{r\textup{-good}}(
   \operatorname{smaller}\{I,J\})\langle TD_If,D_Jf\rangle\\
&=\sum_{i\geq j}\langle T(\phi_i D_i f), D_jg\rangle+\sum_{i<j}\langle TD_if,\phi_jD_jg\rangle \\
&=\sum_{i,j}\langle T(\chi_{ij}D_if),\psi_{ij}D_jg\rangle. \qedhere
\end{align*}
\end{proof}

%
%

We split the subsequent analysis of the series into four cases, depending on whether $i\geq j$ or $i<j$, and whether $\abs{i-j}\leq r$ or $\abs{i-j}>r$. Since the cases $i>j$ and $i<j$ are dual to each other, we only explicitly deal with $i\geq j$, which still splits into the two case $0\leq i-j\leq r$ and $j<i-r$.

\begin{proposition}\label{case1}
Let $T$ be a Calder\'on-Zygmund operator, $f,g\in L^2(\mu)$ as in Section~\ref{sec:finitary}, and $r\in\mathbb{N}$ such that $r\geq 2$. Then
\begin{equation*}
\sum_{\substack{i,j \\ 0\leq i-j \leq r}} \langle T(\chi_{ij}D_i f),\psi_{ij}D_j g\rangle \\
   = \sum_{m\in\mathbb{Z}^d}\sum_H\langle TD_H^{(r,good)}f,P_{H\dot{+}m}^{(r+1)}g\rangle
\end{equation*}
\end{proposition}

\begin{remark}\label{rem:sym-break}
A similar argument would show that
\begin{equation*}
\sum_{\substack{i,j \\ 0< j-i \leq r}} \langle T(\chi_{ij}D_i f),\psi_{ij}D_j g\rangle \\
   = \sum_{m\in\mathbb{Z}^d}\sum_H \langle TP_{H\dot{+}m}^{(r)}f,D_{H}^{(r,good)}g\rangle,
\end{equation*}
where the slight symmetry-break ($P^{(r)}_{H\dot+m}$ vs. $P^{(r+1)}_{H\dot+m}$) is caused by the fact that the diagonal $i=j$ is included in only one of the cases.
\end{remark}


\begin{proof}[Proof of Proposition~\ref{case1}]
\begin{align*}
\sum_{0\leq i-j\leq r}\langle T(\phi_i D_i f),D_jg\rangle&=\sum_{0\leq k\leq r}\sum_{I \ r\textup{-good}}\sum_{J:2^{r-k}\ell(I)=\ell(J)}\langle TD_If,D_Jg\rangle\\
&=\sum_{m\in\mathbb{Z}^d}\sum_{0\leq k\leq r}\sum_H\sum_{\substack{I: I^{(r)}=H\\I \ r\textup{-good}}}\sum_{J:J^{(k)}=H\dot{+}m}\langle TD_If,D_Jg\rangle\\
&=\sum_{m\in\mathbb{Z}^d}\sum_{0\leq k\leq r}\sum_H\langle TD_H^{(r,good)}f,D_{H\dot{+}m}^{(k)}g\rangle\\
&=\sum_{m\in\mathbb{Z}^d}\sum_H\langle TD_H^{(r,good)}f,P_{H\dot{+}m}^{(r+1)}g\rangle.\qedhere
\end{align*}
\end{proof}

\begin{proposition}\label{case2acase2b}
Let $T$ be a Calder\'on-Zygmund operator, $f,g\in L^2(\mu)$ as in Section~\ref{sec:finitary}, and $r\in\mathbb{N}$ such that $r\geq 2$. Then
\begin{align*}
\sum_i\sum_{j<i-r} &\langle T(\chi_{ij}D_i f),\psi_{ij}D_j g\rangle \\
& = \sum_{\substack{m\in\mathbb{Z}^d \\ m\neq 0}}\sum_H\langle TD_H^{(r,good)}f,1_{H\dot{+}m}\rangle\Big(\langle g\rangle_{H\dot{+}m}-\langle g\rangle_H\Big)\\
&\qquad+\sum_H\langle f,D_H^{(r,good)}T^*1\rangle\Big(\langle g\rangle_H-E_{-\infty}g\Big).
\end{align*}
where
$\displaystyle
E_{-\infty}g:=\begin{cases}
\tfrac{1}{\mu(\mathbb{R}^d)}\int g(x)d\mu(x) & \text{if } \mu(\mathbb{R}^d)<\infty\\
0 & \text{otherwise. }
\end{cases}$
\end{proposition}


\begin{proof}
Before starting we want to point out that $\sum_{j<i-r} D_jg=E_{i-r}g-E_{-\infty}g$. Once this remark have been done we can proceed to the proof.
\begin{align*}
\sum_i\sum_{j<i-r} &\langle T(\chi_{ij}D_i f),\psi_{ij}D_j g\rangle =\sum_i\sum_{j<i-r}\langle T\phi_iD_if,D_jg\rangle\\
&=\sum_i\langle T\phi_iD_if,\sum_{j<i-r}D_jg\rangle
=\sum_i\langle T\phi_iD_if,E_{i-r}g-E_{-\infty}g\rangle.
\end{align*}
Recalling the definition of $\phi_i$, the two part of this sum can be written as
\begin{align*}
\sum_i &\langle T\phi_iD_if,E_{-\infty}g\rangle=\sum_i\langle T\phi_iD_if,1\rangle E_{-\infty}g\\
&=\sum_i\langle \phi_iD_if,T^*1\rangle E_{-\infty}g
=\sum_H\langle D_H^{(r,good)}f,T^*1\rangle E_{-\infty}g
\end{align*}
and
\begin{align*}
\sum_i &\langle T\phi_iD_if,E_{i-r}g\rangle=\sum_{\substack{I \ r\textup{-good} \\ J:\ell(J)=2^r\ell(I)}}\langle TD_If,E_Jg\rangle\\
&=\sum_{\substack{H,J \\ \ell(H)=\ell(J)}}\langle T\sum_{\substack{I:I^{(r)}=H \\ I \ r\textup{-good}}}D_If,E_Jg\rangle
=\sum_{m\in\mathbb{Z}^d}\sum_H\langle TD_H^{(r,good)}f,E_{H\dot{+}m}g\rangle\\
&=\sum_{m\in\mathbb{Z}^d}\sum_H\langle TD^{(r,good)}_Hf, 1_{H\dot{+}m}\rangle\big(\langle g\rangle_{H\dot{+}m}-\langle g\rangle_H\big)\\
&\quad \quad +\sum_H\langle TD_H^{(r,good)}f,1\rangle\langle g\rangle_H.
\end{align*}
We can add the restriction $m\neq 0$, since the summand vanishes for $m=0$ in any case.
The self-adjointness of the operator $D_H^{(r,good)}$ finishes the proof.
\end{proof}


The previous two propositions both introduce a double series over terms of the form $\Phi(H,H\dot+m)$. The following lemma provides a useful rearrangement of such summations.

\begin{lemma}\label{lem:insertKgoodness}
\begin{equation*}
  \mathbb{E}\sum_{m\in\Z^d\setminus\{0\}}\sum_H\Phi(H,H\dot+m)
  =\frac{1}{\pi_{\good}}\mathbb{E}\sum_{k=2}^\infty\sum_{\abs{m}\sim 2^{k-2}}\sum_{H\,k\text{-good}}\Phi(H,H\dot+m),
\end{equation*}
where
\begin{equation*}
  \abs{m}\sim 2^{k-2}
  \qquad\overset{\operatorname{def}}{\Leftrightarrow}\qquad 2^{k-3}<\abs{m}\leq 2^{k-2}.
\end{equation*}
\end{lemma}

\begin{proof}
Since every $m\in\Z^d\setminus\{0\}$ satisfies $\abs{m}\sim 2^{k-2}$ for a unique $k\geq 2$, and the $k$-goodness of $H$ is independent of the position of $H$ (and hence of $H\dot+m$), we have
\begin{equation*}
\begin{split}
  \mathbb{E} &\sum_{m\in\Z^d\setminus\{0\}}\sum_{H\in\mathscr{D}^\sigma}\Phi(H,H\dot+m) \\
  &=\frac{1}{\pi_{\good}}\sum_{k=2}^\infty\sum_{\abs{m}\sim 2^{k-2}} \sum_{H\in\mathscr{D}^0} \mathbb{E}(1_{k\text{-good}}(H\dot+\sigma))
     \mathbb{E}\Phi(H\dot+\sigma,H\dot+\sigma\dot+m) \\
  &=\frac{1}{\pi_{\good}}\sum_{k=2}^\infty\sum_{\abs{m}\sim 2^{k-2}} \sum_{H\in\mathscr{D}^0} \mathbb{E}\Big(1_{k\text{-good}}(H\dot+\sigma)
 \Phi(H\dot+\sigma,H\dot+\sigma\dot+m)\Big) \\
   &=\frac{1}{\pi_{\good}}\mathbb{E}\sum_{k=2}^\infty\sum_{\abs{m}\sim 2^{k-2}} \sum_{H\in\mathscr{D}^\sigma\,k\text{-good}} \Phi(H,H\dot+m).\qedhere
\end{split}
\end{equation*}
\end{proof}

The usefulness of the previous rearrangement is in the following:

\begin{lemma}
If $H$ is $k$-good and $\abs{m}\leq 2^{k-2}$, then $H\dot+ m\subset H^{(k)}$.
\end{lemma}


\begin{proof}
Let $K:=H^{(k)}$ so that $\ell(K)=2^k\ell(H)$, and $\operatorname{dist}(H,K^c)\geq\frac14\ell(K)=2^{k-2}\ell(H)$ by definition of $k$-goodness. If $x$ is any interior point of $H$, this means that $\operatorname{dist}(x,K^c)>2^{k-2}\ell(H)$. Every interior point $y$ of $H\dot+m$ has the form $y=x+m\ell(H)$ for such an $x$, and hence $\operatorname{dist}(y,K^c)\geq\operatorname{dist}(x,K^c)-\abs{m}\ell(H)>2^{k-2}\ell(H)-2^{k-2}\ell(H)=0$. Thus $y\in K$ for every interior point $y\in H\dot+m$, and hence $H\dot+m\subset K$.
\end{proof}

A combination of Propositions \ref{case1} and \ref{case2acase2b} with Lemma~\ref{lem:insertKgoodness} shows that
\begin{equation*}
\begin{split}
  \frac{\mathbb{E}}{\pi_{\good}} &\sum_{\substack{i,j \\ j\leq i}} \langle T(\chi_{ij}D_i f),\psi_{ij}D_j g\rangle
  =\frac{\mathbb{E}}{\pi_{\good}} \Big(\sum_{\substack{i,j \\ 0\leq i-j\leq r}}+ \sum_{\substack{i,j \\ j< i-r}}\Big)
      \pair{T(\phi_i D_i f)}{D_j g} \\
  &=\frac{\mathbb{E}}{\pi_{\good}}\sum_H\langle TD_H^{(r,\good)}f,P_H^{(r+1)}g\rangle \\
   &\quad+\frac{\mathbb{E}}{\pi_{\good}^2}\sum_{k=2}^\infty
          \sum_{\abs{m}\sim 2^{k-2}}\sum_{H\,k\text{-good}}
              \langle TD_H^{(r,\good)}f,P_{H\dot+m}^{(r+1)}g\rangle \\
    &\quad +\frac{\mathbb{E}}{\pi_{\good}^2}\sum_{k=2}^\infty
          \sum_{\abs{m}\sim 2^{k-2}}\sum_{H\,k\text{-good}}
              \langle TD_H^{(r,\good)}f,1_{H\dot+m}\rangle  \big(\ave{g}_{H\dot+m}-\ave{g}_H\big) \\
 &\quad     + \frac{\mathbb{E}}{\pi_{\good}}\sum_H D_H^{(r,\good)}T^*1\Big(\ave{g}_H-E_{-\infty}g\Big) \\
  &=:\mathbb{E}\omega(2^{-1})\langle R_1 f,g\rangle
    +\mathbb{E}\sum_{k=2}^\infty\omega(2^{-k}) \langle R_k f,g\rangle \\
  &\qquad +\mathbb{E}\sum_{k=2}^\infty\omega(2^{-k}) \langle Q_k f,g\rangle
    +\mathbb{E}\langle  f, \Pi_{T^*1} g\rangle,
\end{split}
\end{equation*}
where
\begin{equation*}
\begin{split}
  R_1 f &:=\frac{1}{\pi_{\good}}\frac{1}{\omega(\frac12)}\sum_H P_H^{(r+1)}TD_H^{(r,\good)}f, \\
  R_k f &:=\frac{1}{\pi_{\good}}\sum_{\abs{m}\sim 2^{k-2}}
      \sum_{H\,k\text{-good}} \frac{1}{\omega(2^{-k})}P_{H\dot+m}^{(r+1)}TD_H^{(r,\good)}f, \qquad k\geq 2,\\
  Q_k f &:=\frac{1}{\pi_{\good}^2}\sum_{\abs{m}\sim 2^{k-2}}
      \sum_{H\,k\text{-good}} \frac{\langle TD_H^{(r,\good)}f, 1_{H\dot+m}\rangle}{\omega(2^{-k})}\Big(\frac{1_{H\dot+m}}{\mu(H\dot+m)}-\frac{1_H}{\mu(H)}\Big), 
\end{split}
\end{equation*}
and
\begin{equation*}
   \Pi_{T^*1}g :=\sum_H D_H^{(r,\good)}T^*1\Big(\ave{g}_H-E_{-\infty}g\Big).
\end{equation*}
Note that we have incorporated the factor $\omega(2^{-k})$ into the definition of $R_k$ and $Q_k$ in order to achieve a favourable normalisation of the series.

For reasons of symmetry, we also have a similar decomposition of the other half of the double sum, namely
\begin{equation*}
\begin{split}
   \frac{\mathbb{E}}{\pi_{\good}} &\sum_{\substack{i,j \\ i<j}} \langle T(\chi_{ij}D_i f),\psi_{ij}D_j g\rangle
   = \frac{\mathbb{E}}{\pi_{\good}} \sum_{\substack{i,j \\ i<j}} \langle D_i f,T^*(\phi_j D_j g)\rangle \\
   &=\mathbb{E}\omega(2^{-1})\langle f, \tilde{R}_1 f\rangle
    +\mathbb{E}\sum_{k=2}^\infty\omega(2^{-k}) \langle f, \tilde{R}_k g\rangle \\
  &\qquad +\mathbb{E}\sum_{k=2}^\infty\omega(2^{-k}) \langle f, \tilde{Q}_k g\rangle
    +\mathbb{E}\langle \Pi_{T1} f, g\rangle,  
\end{split}
\end{equation*}
where $\tilde{R}_k$ and $\tilde{Q}_k$ have the same form as $R_k$ and $Q_k$, respectively, with the only difference that
\begin{itemize}
  \item $T$ is replaced by $T^*$ throughout, and
  \item $P^{(r+1)}$ in $R_k$ is replaced by $P^{(r)}$ in $\tilde{R}_k$ (cf. Remark~\ref{rem:sym-break}).
\end{itemize}
Defining $Q_1:=0$ and
\begin{equation*}
  R_{-k}:=\tilde{R}_k^*,\qquad Q_{-k}:=\tilde{Q}_k^*,
\end{equation*}
we then obtain the desired identity
\begin{equation*}
  \pair{Tf}{g}
  =\mathbb{E}\sum_{\substack{ k\in\Z \\ k\neq 0}}\omega(2^{-\abs{k}})\Big(\pair{R_k f}{g}+\pair{Q_k f}{g}\Big)
    +\mathbb{E}\Big(\pair{\Pi_{T1}f}{g}+\pair{f}{\Pi_{T^*1}g}\Big).
\end{equation*}
It remains to establish the claimed properties of the operators.

\section{Dyadic representation: estimates}\label{sec:estimates}

Having established the dyadic representation on an algebraic level, we turn to the relevant estimates for the various terms in the obtained expansion. Since the operators $R_k$ and $Q_k$ are essentially similar for positive and negative values of $k$, we only explicitly deal with $k\geq 1$.

\subsection{The paraproducts}

Our dyadic decomposition of the operator $T$ led to two dyadic paraproducts of the form
\begin{equation*}
\begin{split}
  \Pi_b f &=\sum_{H\in\mathscr{D}} D_H^{(r,\good)}b\Big(\ave{f}_H-E_{-\infty}f\Big),\qquad b\in\{T1,T^*1\}, \\
    &=\sum_{H\in\mathscr{D}}\Big(\sum_{\substack{I:I^{(r)}=H \\ \operatorname{dist}(I,H^c)\geq\frac14\ell(H)}}D_I b\Big)\ave{f}_H
       -\Big(\sum_{H\in\mathscr{D}}D_H^{(r,\good)}b\Big)E_{-\infty}f \\
    &=:\Pi_b^1 f-\Pi_b^2 f.
\end{split}
\end{equation*}
These are a standard part of our decomposition, which have been studied in exactly the same form in \cite{NTV}, where the following result is proven:

\begin{proposition}[\cite{NTV}, Theorem 7.1]\label{proofcase2b}
For any $\lambda>1$, we have
\begin{equation*}
  \Norm{\Pi_b^1 f}{2}\leq C\Norm{b}{BMO^2_\lambda(\mu)}\Norm{f}{2},
\end{equation*}
where, for $p\in[1,\infty)$,
\begin{equation*}
  \Norm{b}{BMO^p_\lambda(\mu)}
  :=\sup_Q \inf_{a\in\C} \Big(\frac{1}{\mu(\lambda Q)}\int_Q|f-a|^pd\mu\Big)^{1/p},
\end{equation*}
and the supremum is over all cubes $Q$ in $\R^d$ (including $Q=\lambda Q=\R^d$ if $\mu(\R^d)<\infty$).
\end{proposition}

For the part $\Pi_b^2$, the same estimate is easy:

\begin{lemma}
\begin{equation*}
   \Norm{\Pi_b^2 f}{2}\leq C\Norm{b}{BMO_\lambda^2(\mu)}\Norm{f}{2}.
\end{equation*}
\end{lemma}

\begin{proof}
Note that this term is only nonzero if $\mu(\R^d)<\infty$, and in this case, noting that $D_H^{(r,\good)}a=0$ for any constant $a$,
\begin{equation*}
\begin{split}
  \Norm{\Pi^2_b f}{2}
  &=\BNorm{\sum_H D_H^{(r,\good)}(b-a)}{2}\abs{E_{-\infty}f} \\
  &\leq\Norm{b-a}{2}\cdot\mu(\R^d)^{-1/2}\Norm{f}{2}
  \leq\Norm{b}{BMO_\lambda^2(\mu)}\Norm{f}{2}
\end{split}
\end{equation*}
by a suitable choice of $a$.
\end{proof}

Thus, all that remains is to check the BMO conditions on $b\in\{T1,T^*1\}$ under our assumptions on the operator $T$. This is also reasonably standard, and contained in the following:

\begin{proposition}\label{bmolambda}
Under the assumptions of Theorem~\ref{thm:newRepr} and the standard Dini condition \eqref{eq:logDini} with $\alpha=0$, we have
 $T1\in BMO_{\lambda}^2$ and $T^*1\in BMO_{\lambda}^2$ for any $\lambda>1$.
\end{proposition}

\begin{proof}
We are going to prove that $T1\in BMO_{\lambda}^2$  and the case of the adjoint is similar.
Fix a cube $Q$, and some $\lambda>1$. We denote by $x_Q$ the centre of the cube $Q$.

Note that we have $|x-x_Q|<\frac12\ell(Q)\leq\frac12|y-x|$ for $x\in Q$ and $y\in(2 Q)^c$ (using the $\ell^\infty$ metric on $\R^d$ for convenience.) So for $x\in Q$ and $\tau\geq 2$ we have that
\begin{align*}
|T1_{(\tau Q)^c}(x)-T1_{(\tau Q)^c}(x_Q)|&\leq\int_{(\tau Q)^c}|K(x,y)-K(x_Q,y)|d\mu(y)\\
&\leq C\int_{(2 Q)^c}\omega\Big(\frac{|x-x_Q|}{|y-x_Q|}\Big)\frac{1}{|y-x_Q|^n}d\mu(y)\\
&\leq C\sum_{k=1}^{\infty}\int_{2^{k+1}Q\setminus 2^kQ}\omega\Big(\frac{2^{-1}\ell(Q)}{2^{k-1}\ell(Q)}\Big)\frac{1}{(2^{k-1}\ell(Q))^n}d\mu(y)\\
&\leq C\sum_{k=1}^{\infty}\omega(2^{-k})\frac{\mu(2^{k+1}Q)}{(2^{k-1}\ell(Q))^n} \\
&\leq C \sum_{k=1}^{\infty}\omega(2^{-k})
 \leq C\int_0^1\omega(t)\frac{dt}{t}.
\end{align*}
Additionally, if $x\in Q$ and $\lambda\in(1,2)$
\begin{align*}
|T1_{2Q\setminus\lambda Q}(x)|&\leq \int_{2Q\setminus\lambda Q}|K(x,y)|d\mu(y)\\
&\leq C\int_{2Q\setminus\lambda Q}\frac{1}{|x-y|^n}d\mu(y)\\
&\leq C\int_{2Q\setminus\lambda Q}\frac{1}{(\lambda-1)^n\ell(Q)^n}d\mu(y)\\
&\leq C\frac{1}{(\lambda-1)^n}.
\end{align*}

Set $\tau=\max\{2,\lambda\}$ and $a_Q=T1_{(\tau Q)^c}(x_Q)$. Then
\begin{align*}
\int_Q &|T1(x)-a_Q|^2d\mu(x) \\
&\leq C\int_{\mathbb{R}^d}|T1_{\lambda Q}(x)|^2d\mu(x)+C\int_Q|T1_{\tau Q\setminus\lambda Q}(x)|^2d\mu(x)\\
&\phantom{\leq C\int_{\mathbb{R}^d}|T1_{\lambda Q}}+C\int_Q|T1_{(\tau Q)^c}(x)-T1_{(\tau Q)^c}(x_Q)|^2d\mu(x)\\
&\leq C\mu(\lambda Q)+C\frac{1}{(\lambda-1)^n}\mu(Q)+C\Big(\sum_{k=0}^{\infty}\omega(2^{-k})\Big)^2\mu(Q)\\
&\leq C\mu(\lambda Q),
\end{align*}
where the middle term of the decomposition is equal to zero for $\lambda\geq 2$.
\end{proof}

\subsection{The operator $R_1$}

We have
\begin{equation*}
  R_1=\sum_H B_H^{(1)},
\end{equation*}
where the operator $B_H^{(1)}$ given by
\begin{equation}\label{eq:BH1}
\begin{split}
  \pi_{\good}\omega(\tfrac12)B_H^{(1)}f &:=P_H^{(r+1)}TD_H^{(r,\good)}f
  =\sum_{\substack{ I\,r\text{-good} \\ I^{(r)}=H}}\sum_{J:J\subseteq H\subseteq J^{(r)}}
     D_J TD_If \\
  &=\sum_{\substack{ I\,r\text{-good} \\ I^{(r)}=H}}\sum_{J:J\subseteq H\subseteq J^{(r)}}\sum_{i,j}
     \pair{ T\varphi_I^i }{ \varphi_J^j} \pair{\varphi_I^i}{f}\varphi_J^j.
\end{split}
\end{equation}

\begin{lemma}
\begin{equation*}
  \Norm{T\varphi_I^i}{2}\leq C.
\end{equation*}
\end{lemma}

\begin{proof}
Since $\varphi_I^i$ takes a constant value $\ave{\varphi_I^i}_{I'}$ on each $I'\in\operatorname{ch}(I)$, it follows from the testing condition $\Norm{T1_I}{2}\leq C\mu(I)^{1/2}$ that
\begin{equation*}
\begin{split}
  \Norm{T\varphi_I^i}{2}
  &=\BNorm{T\sum_{I'\in\operatorname{ch}(I)}\ave{\varphi_I^i}_{I'}1_{I'}}{2}
  \le\sum_{I'\in\operatorname{ch}(I)}\abs{\ave{\varphi_I^i}_{I'}}\Norm{T1_{I'}}{2} \\
  &\le C\sum_{I'\in\operatorname{ch}(I)}\abs{\ave{\varphi_I^i}_{I'}}\mu(I')^{1/2} 
    \le C2^{d/2}\Big(\sum_{I'\in\operatorname{ch}(I)}\abs{\ave{\varphi_I^i}_{I'}}^2\mu(I')\Big)^{1/2} \\
  &\leq C\Norm{\varphi_I^i}{2}=C,
\end{split}
\end{equation*}
where we incorporated the dimensional factor $2^{d/2}$ into $C$.
\end{proof}

It follows that
\begin{equation*}
\begin{split}
  \abs{\pair{B_H^{(1)}f}{g}}
  &\leq C\sum_{\substack{ I\,r\text{-good} \\ I^{(r)}=H}}\sum_{J:J\subseteq H\subseteq J^{(r)}}\sum_{i,j}
  \abs{\pair{\varphi_I^i}{f}}\abs{\pair{\varphi_J^j}{g}} \\
  &\leq C\Big(\sum_{I:I^{(r)}=H}\sum_i\abs{\pair{\varphi_I^i}{f}}^2\Big)^{1/2}
  \Big(\sum_{J:J\subseteq H\subseteq J^{(r)}}\sum_j\abs{\pair{\varphi_I^i}{f}}^2\Big)^{1/2} \\
  &=C\Norm{D_H^{(r)}f}{2}\Norm{P_H^{(r+1)}g}{2}
\end{split}
\end{equation*}
using Caucy--Schwarz and the fact that the total number of terms is bounded by a dimensional constant.
Hence
\begin{equation*}
\begin{split}
  \abs{\pair{R_1f}{g}}
  &\leq\sum_H\abs{\pair{B_H^{(1)}f}{g}} \\
  &\leq C\Big(\sum_H\Norm{D_H^{(r)}f}{2}^2\Big)^{1/2}\Big(\sum_H\Norm{P_H^{(r+1)}g}{2}\Big)^{1/2} \\
  & \leq C\Norm{f}{2}\sqrt{r+1}\Norm{g}{2}\leq C\Norm{f}{2}\Norm{g}{2},
\end{split}
\end{equation*}
incorporating the fixed constant $\sqrt{1+r}$ into $C$,

From \eqref{eq:BH1}, the kernel $b_H^{(1)}$ of $B_H^{(1)}$ is given by
\begin{equation*}
  b_H^{(1)}(x,y)
  =\frac{1}{\pi_{\good}\omega(\tfrac12)}\sum_{\substack{ I\,r\text{-good} \\ I^{(r)}=H}}\sum_{J:J\subseteq H\subseteq J^{(r)}}\sum_{i,j}
     \pair{ T\varphi_I^i }{ \varphi_J^j} \varphi_I^i(y)\varphi_J^j(x).
\end{equation*}
In general, there are no good pointwise bounds for this expression; however, if $\mu$ is a doubling measure, then
\begin{equation*}
  \abs{\varphi_I^i(y)}
  \leq\frac{C}{\mu(I)^{1/2}}1_I(y)\leq\frac{C}{\mu(H)^{1/2}}1_I(y)
\end{equation*}
with a similar bound for $\abs{\varphi_J^j(x)}$. Thus
\begin{equation*}
\begin{split}
  \abs{b^{(1)}_H(x,y)}
  &\leq C\sum_{\substack{ I\,r\text{-good} \\ I^{(r)}=H}} \frac{1_I(y)}{\mu(H)^{1/2}}\sum_{J:J\subseteq H\subseteq J^{(r)}} \frac{1_J(x)}{\mu(H)^{1/2}}
      \sum_{i,j}1 \\
  &\leq \frac{C}{\mu(H)}1_H(x)1_H(y).
\end{split}
\end{equation*}

\subsection{The operators $R_k$, $k\geq 2$}

For the analysis of $R_k$ (as well as $Q_k$ below), it is convenient to take as the new summation variable $K:=H^{(k)}$, which is a common ancestor of both $H$ and $H\dot+m$ for $k$-good $H$ and $\abs{m}\sim 2^{k-2}$. This leads to the decomposition
\begin{equation*}
\begin{split}
  R_k f
  &=\sum_K\frac{1}{\pi_{\good}}\sum_{\abs{m}\sim 2^{k-2}}\sum_{\substack{H\,k\text{-good} \\ H^{(k)}=K }}
        \frac{1}{\omega(2^{-k})}P_{H\dot+m}^{(r+1)}TD_H^{(r,\good)}f
        =:\sum_K B^{(k)}_K f,
\end{split}
\end{equation*}
where
\begin{equation*}
  B_K^{(k)}
  =\frac{1}{\pi_{\good}}\sum_{\abs{m}\sim 2^{k-2}}\sum_{\substack{H\,k\text{-good} \\ H^{(k)}=K }}
  \sum_{\substack{I\,r\text{-good} \\ I^{(r)}=H}}\sum_{\substack{J:J\subseteq H\dot+m \\ \ \subseteq J^{(r)}}}
        \frac{1}{\omega(2^{-k})}D_JTD_I
\end{equation*}
has kernel
\begin{equation*}
  b_K^{(k)}(x,y)
  =\frac{1}{\pi_{\good}}\sum_{\abs{m}\sim 2^{k-2}}\sum_{\substack{H\,k\text{-good} \\ H^{(k)}=K }}
  \sum_{\substack{I\,r\text{-good} \\ I^{(r)}=H}}\sum_{\substack{J:J\subseteq H\dot+m \\ \ \subseteq J^{(r)}}}\sum_{i,j}
        \frac{\pair{T\varphi_I^i}{\varphi_J^j}}{\omega(2^{-k})}\varphi_I^i(y)\varphi_J^j(x).
\end{equation*}

\begin{lemma}\label{lem:RkKernel1}
In the sum above, we have
\begin{equation*}
  \Babs{\frac{\pair{T\varphi_I^i}{\varphi_J^j}}{\omega(2^{-k})}\varphi_I^i(y)\varphi_J^j(x)}
  \leq\frac{C}{\ell(K)^n}1_I(y)1_J(x).
\end{equation*}
\end{lemma}

\begin{proof}
Since $I$ is $r$-good and $I^{(r)}=H$, it follows from the definition that $\dist(I,H^c)\geq\frac14\ell(H)$. Since $J\subset H\dot+m$ and $\abs{m}\sim 2^{k-2}$, we further deduce that $\dist(I,J)\gtrsim \abs{m}\ell(H)\gtrsim 2^k\ell(H)=\ell(K)$. Denoting by $c_I$ the centre of $I$ and using the vanishing integral of $\varphi_I^i$ and kernel regularity, we obtain
\begin{align*}
\left|\langle T\varphi_I^{i},\varphi_J^{j}\rangle\right|
&=\left|\iint K(x,y)\varphi_I^{n_1}(y)\varphi_J^{n_2}(x)d\mu(y)d\mu(x)\right|\\
&=\left|\iint (K(x,y)-K(x,c_I))\varphi_I^{i}(y)\varphi_J^{j}(x)d\mu(y)d\mu(x)\right|\\
&\leq\iint C\omega\Big(\frac{\abs{y-c_I}}{\abs{x-c_I}}\Big)\frac{1}{\abs{x-c_I}^d}
     |\varphi_I^{n_1}(y)|\,|\varphi_J^{n_2}(x)|d\mu(y)d\mu(x)\\
&\leq C\frac{\omega(2^{-k})}{\ell(K)^n}\Vert \varphi_I^{i}\Vert_1\Vert \varphi_J^{j}\Vert_1.
\end{align*}
The proof is completed by recalling that $\varphi_I^i$ is supported on $I$ and
\begin{equation*}
  \Norm{\varphi_I^i}{1}\Norm{\varphi_I^i}{\infty}\leq C,
\end{equation*}
with similar observations for $\varphi_J^j$.
\end{proof}

\begin{lemma}\label{lem:BkNorm}
The kernel $b_K^{(k)}$ of $B_K^{(k)}$ satisfies
\begin{equation*}
  \abs{b_K^{(k)}(x,y)}\leq \frac{C}{\ell(K)^n}1_K(x)1_K(y),
\end{equation*}
and hence
\begin{equation*}
  \abs{\pair{B_K^{(k)}f}{g}}
  \leq\frac{C}{\ell(K)^n}\Norm{1_K f}{1}\Norm{1_K g}{1}
  \leq C\Norm{f}{2}\Norm{g}{2}.
\end{equation*}
\end{lemma}

\begin{proof}
From the previous lemma, we conclude that
\begin{equation*}
\begin{split}
  \abs{b_K^{(k)}(x,y)}
  &\leq\frac{C}{\ell(K)^n}
  \sum_{\abs{m}\sim 2^{k-2}}\sum_{\substack{H\,k\text{-good} \\ H^{(k)}=K }}
  \sum_{\substack{I\,r\text{-good} \\ I^{(r)}=H}}\sum_{\substack{J:J\subseteq H\dot+m \\ \ \subseteq J^{(r)}}}1_I(y)1_J(x) \\
  &\leq\frac{C}{\ell(K)^n}
  \sum_{\abs{m}\sim 2^{k-2}}\sum_{\substack{H\,k\text{-good} \\ H^{(k)}=K }}1_H(y)1_{H\dot+m}(x), \\
\end{split}
\end{equation*}
since the cubes $I:I^{(r)}=H$ are pairwise disjoint, and the cubes $J:J\subseteq H\dot+m\subseteq J^{(r)}$ have overlap at most $r+1$ times at any point, and we have incorporated this fixed constant into $C$ above.

Consider a fixed pair $(x,y)$. Then there is at most one $H:H^{(k)}=K$ such that $H\owns y$. For this $H$, there is at most one integer $m$ such that $H\dot+m\owns x$. So altogether there is at most one non-zero summand above. Since both $H$ and $H\dot+m$ are subsets of $K$, the non-zero summand can only exist is $(x,y)\in K\times K$. This proves the required kernel bound.

In the operator bound, the first estimate is immediate, and the second one follows by $\Norm{1_K f}{1}\leq\mu(K)^{1/2}\Norm{f}{2}$ (applied to $g$ as well) and $\mu(K)\leq C\ell(K)^n$.
\end{proof}

\begin{lemma}
\begin{equation*}
  \abs{\pair{R_k f}{g}}\leq C\Norm{f}{2}\Norm{g}{2}.
\end{equation*}
\end{lemma}

\begin{proof}
In the summation defining $B_K^{(k)}$, we observe that $I^{(r)}=H$ and $H^{(k)}=K$, hence $I^{(r+k)}=K$, and $J^{(j)}=H\dot+m$ for some $j=0,\ldots,r$ and $(H\dot+m)^{(k)}=K$, hence $J^{(j+k)}=K$ for some $j=0,\ldots,r$. It follows that
\begin{equation*}
  B_K^{(k)}=\sum_{j=0}^r D_K^{(k+j)}B_K^{(k)}D_K^{(k+r)}
  =(P_K^{(k+r+1)}-P_K^{(k)})B_K^{(k)}D_K^{(k+r)},
\end{equation*}
and hence
\begin{equation*}
\begin{split}
   \abs{\pair{R_k f}{g}}
  & \leq\sum_K \abs{\pair{(B_K^{(k)}D_K^{(k+r)}f}{(P_K^{(k+r+1)}-P_K^{(k)})g}} \\
  &\leq C\Big(\sum_K\Norm{D_K^{(k+r)}f}{2}^2\Big)^{1/2}
  \Big(\sum_K\Norm{(P_K^{(k+r+1)}-P_K^{(k)})g}{2}^2\Big)^{1/2} \\
  & \leq C\Norm{f}{2}\Big(\sum_{j=0}^r\sum_K\Norm{D_K^{(k+j)}g}{2}^2\Big)^{1/2} \\
  &\leq C\Norm{f}{2}\sqrt{r+1}\Norm{g}{2}
  \leq C\Norm{f}{2}\Norm{g}{2}.\qedhere
\end{split}
\end{equation*}
\end{proof}

\subsection{The operators $Q_k$}

Noting that $Q_1:=0$ trivially satisfies the required estimates, we concentrate on $Q_k$ with $k\geq 2$.

Taking the new summation variable $K:=H^{(k)}$ as for $R_k$, we are led to the decomposition
\begin{equation*}
\begin{split}
  Q_k f &=\sum_K\sum_{\abs{m}\sim 2^{k-2}}
      \sum_{\substack{ H\,k\text{-good} \\ H^{(k)}=K }}
      \frac{\langle TD_H^{(r,\good)}f, 1_{H\dot+m}\rangle}{\pi_{\good}^2\omega(2^{-k})}\Big(\frac{1_{H\dot+m}}{\mu(H\dot+m)}-\frac{1_H}{\mu(H)}\Big) \\
      &=:\sum_K A_K^{(k)}f.
\end{split}
\end{equation*}

\begin{lemma}
\begin{equation*}
    A_K^{(k)}
    =P^{(k+1)}_KA_K^{(k)}D_K^{(k+r)}
\end{equation*}
\end{lemma}

\begin{proof}
Expanding
\begin{equation*}
  D_H^{(r,\good)}=\sum_{\substack{ I\,r\text{-good} \\ I^{(r)}=H}}D_I
\end{equation*}
and noting that $I^{(r)}=H$ and $H^{(k)}=K$ imply $I^{(r+k)}=K$, the identity $A_K^{(k)}=D_K^{(k+r)}$ is immediate.

Concerning the post-composition of $A_K^{(k)}$ by $P_K^{(k+1)}$, we make the following observations. First, $A_K^{(k)}f$ is supported on $K$, which again depends on the fact that the $k$-goodness of $H$ together with $\abs{m}\sim 2^{k-2}$ imply that $H\dot+m$ is contained in $H^{(k)}=K$. Second, $A_K^{(k)}f$ is constant on the $k$th order descendant of $K$, which is immediate from its expression as a superposition of the relevant indicator functions. Thirs, $\int A_K^{(k)}f d\mu=0$, which is also immediate from the fact that this property clearly holds for each of the summands,
\begin{equation*}
  \int\Big(\frac{1_{H\dot+m}}{\mu(H\dot+m)}-\frac{1_H}{\mu(H)}\Big)d\mu=0.
\end{equation*}
These three properties characterise the range of the projection $P_K^{(k+1)}$, and hence $A_K^{(k)}=P_K^{(k+1)}A_K^{(k)}$.
\end{proof}

The kernel of $A_K^{(k)}$ is given by
\begin{equation*}
\begin{split}
  & a_K^{(k)}(x,y) \\
  &=\sum_{\abs{m}\sim 2^{k-2}}
      \sum_{\substack{ H\,k\text{-good} \\ H^{(k)}=K }}\sum_{\substack{ I\,r\text{-good} \\ I^{(r)}=H }}\sum_i
      \frac{\langle T\varphi_I^i, 1_{H\dot+m}\rangle}{\pi_{\good}^2\omega(2^{-k})}
         \varphi_I^i(y)\Big(\frac{1_{H\dot+m}(x)}{\mu(H\dot+m)}-\frac{1_H(x)}{\mu(H)}\Big) \\
   &=:a_{K,1}^{(k)}(x,y)-a_{K,2}^{(k)}(x,y),
\end{split}
\end{equation*}
where the last two kernels are defined in a natural way by taking all the terms of the form $1_{H\dot+m}(x)/\mu(H\dot+m)$ to the first one, and all those of the form $1_{H}(x)/\mu(H)$ to the second.

\begin{lemma}\label{lem:AkKernel1}
In the sum above, we have
\begin{equation*}
  \Babs{ \frac{\langle T\varphi_I^i, 1_{H\dot+m}\rangle}{\omega(2^{-k})}\varphi_I^i(y)}
  \leq \frac{C}{\ell(K)^n}\mu(H\dot+m)1_I(y).
\end{equation*}
\end{lemma}

\begin{proof}
By essentially the same considerations as in the beginning of the proof of Lemma~\ref{lem:RkKernel1}, we check that $\dist(I,H\dot+m)\gtrsim 2^k\ell(H)=\ell(K)$. As in the same proof, using the vanishing integral of $\varphi_I^i$ to insert $K(x,c_I)$, where $c_I$ is the centre of $I$, we hace
\begin{align*}
\left|\langle T\varphi_I^{n_1},1_{H\dot{+}m}\rangle\right|&=\left|\iint K(x,y)\varphi_I^{i}(y)1_{H\dot{+}m}(x)d\mu(y)d\mu(x)\right|\\
&=\left|\iint (K(x,y)-K(x,c_I))\varphi_I^{i}(y)1_{H\dot{+}m}(x)d\mu(y)d\mu(x)\right| \\
&\leq\iint \omega\Big(\frac{\abs{y-c_I}}{\abs{x-c_I}}\Big)\frac{C}{\abs{x-c_I}^d}|\varphi_I^{n_1}(y)|\, 1_{H\dot{+}m}(x)d\mu(y)d\mu(x)\\
&\leq C\omega(2^{-k})\frac{\Vert\varphi_I^{n_1}\Vert_1\mu(H\dot{+}m)}{\ell(K)^n}.
\end{align*}
The proof is completed by recalling that $\varphi_I^i$ is supported on $I$, and using the estimate $\Norm{\varphi_I^i}{1}\Norm{\varphi_I^i}{\infty}\leq C$.
\end{proof}

\begin{lemma}
\begin{equation*}
    \abs{a_{K,1}^{(k)}(x,y)}
    \leq \frac{C}{\ell(K)^n}1_K(x)1_K(y)
\end{equation*}
\end{lemma}

\begin{proof}
From the previous lemma and the disjointness of $I:I^{(r)}=H$, we have
\begin{equation*}
\begin{split}
  \abs{a_{K,1}^{(k)}(x,y)}
    &\leq \sum_{\abs{m}\sim 2^{k-2}}
      \sum_{\substack{ H\,k\text{-good} \\ H^{(k)}=K }}\frac{C}{\ell(K)^n}1_H(y)1_{H\dot+m}(x).
\end{split}
\end{equation*}
For a pair $(x,y)$, there is at most one $H:H^{(k)}=K$ such that $H\owns y$, and fixing this $H$, there is at most one integer $m$ such that $H\dot+m\owns x$. Moreover, $(x,y)\in K\times K$ is a necessary condition for the existence of such $H$ and $m$. This proves the asserted bound.
\end{proof}

\begin{lemma}
\begin{equation*}
    \abs{a_{K,2}^{(k)}(x,y)}
    \leq C\sum_{H:H^{(k)}=K}\frac{1_H(x)1_H(y)}{\mu(H)}.
\end{equation*}
\end{lemma}

\begin{proof}
By the same initial considerations as in the previous lemma, we have
\begin{equation*}
    \abs{a_{K,2}^{(k)}(x,y)}
    \leq \sum_{\abs{m}\sim 2^{k-2}}
      \sum_{\substack{ H\,k\text{-good} \\ H^{(k)}=K }}\frac{C}{\ell(K)^n}\mu(H\dot+m)1_H(y)\frac{1_{H}(x)}{\mu(H)}
\end{equation*}
For each fixed $H$, we observe that the cubes $H\dot+m$, with $\abs{m}\sim 2^{k-2}$, are pairwise disjoint and contained in $K$, thus
\begin{equation*}
  \sum_{\abs{m}\sim 2^{k-2}}\mu(H\dot+m)
  \leq\mu(K)\leq C\ell(K)^n.\qedhere
\end{equation*}
\end{proof}

\begin{lemma}
\begin{equation*}
   \abs{\pair{A_K^{(k)}f}{g}}\leq C\Norm{f}{2}\Norm{g}{2}.
\end{equation*}
\end{lemma}

\begin{proof}
Based on the kernel bounds for the two part of $A_K^{(k)}$, it follows that
\begin{equation*}
  \abs{\pair{A_K^{(k)}f}{g}}
  \leq\frac{C}{\ell(K)^n}\Norm{1_K f}{1}\Norm{1_K g}{1}
  +C\sum_{H:H^{(k)}=K}\frac{\Norm{1_H f}{1}\Norm{1_H g}{1}}{\mu(H)}
\end{equation*}
The first term has the desired bound by the easy argument that we already gave in Lemma~\ref{lem:BkNorm}. Concerning the second term, two applications of Cauchy--Schwarz give
\begin{equation*}
\begin{split}
  \sum_{H:H^{(k)}=K} &\frac{\Norm{1_H f}{1}\Norm{1_H g}{1}}{\mu(H)}
  \leq\sum_{H:H^{(k)}=K}\Norm{1_H f}{2}\Norm{1_H g}{2} \\
  &\leq\Big(\sum_{H:H^{(k)}=K}\Norm{1_H f}{2}^2\Big)^{1/2}
    \Big(\sum_{H:H^{(k)}=K}\Norm{1_H g}{2}^2\Big)^{1/2} 
   =\Norm{1_K f}{2}\Norm{1_K g}{2},
\end{split}
\end{equation*}
and this completes the proof.
\end{proof}

\begin{lemma}
\begin{equation*}
    \abs{\pair{Q_k f}{g}}
    \leq C\sqrt{k}\Norm{f}{2}\Norm{g}{2}.
\end{equation*}
\end{lemma}

\begin{proof}
\begin{equation*}
\begin{split}
    \abs{\pair{Q_k f}{g}}
    &\leq\sum_K\abs{\pair{A_K D_K^{(k+r)}f}{P_K^{(k+1)}g}} \\
    &\leq C\Big(\sum_K \Norm{D_K^{(k+r)}f}{2}^2\Big)^{1/2}\Big(\sum_K\Norm{P_K^{(k+1)}g}{2}^2\Big)^{1/2} \\
    &\leq C\Norm{f}{2}\Big(\sum_{j=0}^k\sum_K\Norm{D_K^{(j)}g}{2}^2\Big)^{1/2} \\
    &\leq C\Norm{f}{2}\sqrt{k+1}\Norm{g}{2}
      \leq C\sqrt{k}\Norm{f}{2}\Norm{g}{2},
\end{split}
\end{equation*}
using $k\geq 2$ in the last step.
\end{proof}

We have now completed the proof of Theorem~\ref{thm:newRepr}, except for the last claim concerning the representation of the new operators $Q_k$ and $R_k$ as dyadic shits. Note that this is already enough for the deduction of Corollary~\ref{cor:T1}.

\section{The shift structure of the new operators}\label{sec:shifts}

A dyadic shift of type $(u,v)$ can be defined as an operator of the form
\begin{equation*}
  S_{u,v}
  =\sum_{K\in\mathscr{D}}C_K^{(u,v)},
\end{equation*}
where the operators $C_K^{(u,v)}$ satisfy the orthogonality property
\begin{equation*}
  C_K^{(u,v)}=D_K^{(u)}C_K^{(u,v)}D_K^{(v)}
\end{equation*}
and their kernels $c_K^{(u,v)}$ have the pointwise bound
\begin{equation*}
  \abs{c_K^{(u,v)}(x,y)}\leq C\frac{1_K(x)1_K(y)}{\mu(K)}.
\end{equation*}
Since
\begin{equation*}
  D_K^{(u)}f=\sum_{I:I^{(u)}=K}\sum_i \varphi_I^i\pair{\varphi_I^i}{f}
\end{equation*}
and $\Norm{\varphi_I^i}{1}\leq C\mu(I)^{1/2}$, this is easily seen to coincide with the definition in terms of Haar functions that is used in several papers.

\subsection{The operators $R_k$}

We concentrate on $k\geq 1$, since the case of negative $k$ is similar and essentially dual to this one. (Note that the adjoint of a dyadic shift of type $(u,v)$ is a shift of type $(v,u)$.) Then
\begin{equation*}
  R_k=\sum_{K\in\mathscr{D}}B_K^{(k)},\qquad
  B_K^{(k)}=\sum_{j=k}^{k+r}D_K^{(j)}B_K^{(k)}D_K^{(k+r)},
\end{equation*}
and the kernel $b_K^{(k)}$ of $B_K^{(k)}$ satisfies
\begin{equation*}
  \abs{b_K^{(k)}(x,y)}\leq C\frac{1_K(x)1_K(y)}{\ell(K)^n}
  \leq C\frac{1_K(x)1_K(y)}{\mu(K)}.
\end{equation*}
Since $D_K^{(j)}$ is a difference of averaging operators (or by the pointwise bounds for the Haar functions), it follows that the kernel of $D_K^{(j)}B_K^{(k)}$ satisfies the same bound. Thus, letting
\begin{equation*}
  C_K^{(j,k+r)}:=D_K^{(j)}B_K^{(k)},\qquad S_{j,k+r}:=\sum_K C_K^{(j,k+r)},
\end{equation*}
it is immediate that $S_{j,k+r}$ is a shift of type $(j,k+r)$ and
\begin{equation*}
  R_k=\sum_{j=k}^{k+r}S_{j,k+r}
\end{equation*}
is a sum of $r+1$ shifts of complexity at most $k+r$.

\subsection{The operators $Q_k$}

For $k\geq 1$ again, recall that
\begin{equation*}
   Q_k=\sum_{K\in\mathscr{D}}A_K^{(k)},\qquad A_K^{(k)}=\sum_{j=0}^k D_K^{(j)}A_K^{(k)}D_K^{(k+r)},
\end{equation*}
where the kernel $a_K^{(k)}$ of $A_K^{(k)}$ satisfies
\begin{equation*}
  \abs{a_K^{(k)}(x,y)}
  \leq C\frac{1_K(x)1_K(y)}{\ell(K)^n}
    +C\sum_{H:H^{(k)}=K}\frac{1_H(x)1_H(y)}{\mu(H)}.
\end{equation*}
Then we can split
\begin{equation*}
  A_K^{(k)}
  =A_K^{(k.1)}+\sum_{H:H^{(k)}=K}A_H^{(k,2)},
\end{equation*}
where the corresponding kernels $a_K^{(k,1)}$ and $a_H^{(k,2)}$ satisfy
\begin{equation*}
  \abs{a_K^{(k,1)}(x,y)}
  \leq C\frac{1_K(x)1_K(y)}{\mu(K)},\qquad
   \abs{a_H^{(k,2)}(x,y)}\leq C\frac{1_H(x)1_H(y)}{\mu(H)}.
\end{equation*}
Thus
\begin{equation}\label{eq:decompQk}
  Q_k
  =\sum_{j=0}^{k}\sum_K D_K^{(j)}\Big(A_{K,1}^{(k)}+\sum_{H:H^{(k)}=K}A_H^{(k,2)}\Big)D_K^{(k+r)},
\end{equation}
and it is immediate that each
\begin{equation*}
   \sum_K D_K^{(j)}A_{K,1}^{(k)}D_K^{(k+r)}
\end{equation*}
is a dyadic shift of order $(j,k+r)$. We will prove that

\begin{lemma}\label{lem:difficultShift}
\begin{equation*}
  \sum_K D_K^{(j)}\sum_{H:H^{(k)}=K}A_H^{(k,2)}D_K^{(k+r)}
\end{equation*}
is a dyadic shift of order $(0,k-j+r)$.
\end{lemma}

Once this is proven, it follows from \eqref{eq:decompQk} that $Q_k$ is a sum of $2(k+1)$ shift of complexity at most $k+r$.

\begin{proof}[Proof of Lemma~\ref{lem:difficultShift}]
As written, the series looks formally more like a shift of order $(j,k+r)$, but the the kernel of $\sum_{H:H^{(k)}=K}A_H^{(k,2)}$ does not have the correct bound. Thus we need to reorganise the summation:
\begin{equation*}
\begin{split}
  \sum_K & D_K^{(j)}\sum_{H:H^{(k)}=K}A_H^{(k,2)} D_K^{(k+r)} \\
  &=\sum_K\Big(\sum_{J:J^{(j)}=K}D_J\Big)\Big(\sum_{L:L^{(j)}=K}\sum_{H:H^{(k-j)}=L}A_H^{(k.2)}\Big)\Big(\sum_{M:M^{(j)}=K}D_M^{(k-j+r)}\Big) \\
  &=\sum_J D_J\sum_{H:H^{(k-j)}=J}A_H^{(k,2)}D_J^{(k-j+r)},
\end{split}
\end{equation*}
using support considerations to see that only the terms with $M=L=J$ are non-zero, and simplifying $\sum_K\sum_{J:J^{(k)}=K}=\sum_J$.

Now, the right hand side has the structure of a shift of order $(0,k-j+r)$, and we only need to verify the bound for the kernel. The kernel of $D_J\sum_H A_H^{(k,2)}$ is given by
\begin{equation*}
\begin{split}
 &\Babs{ \sum_i\varphi_J^i(x)\int\varphi_J^i(z)\sum_{H:H^{(k-j)}=J}a_H^{(k,2)}(z,y)d\mu(z)} \\
 &\leq C\sum_i \frac{1_J(x)}{\mu(J)^{1/2}}
    \sum_{H:H^{(k-j)}=J}\int \frac{1_J(z)}{\mu(J)^{1/2}}\frac{1_H(z)1_H(y)}{\mu(H)}d\mu(z) \\
  &\leq C\frac{1_J(x)}{\mu(J)}\sum_{H:H^{(k-j)}=J}1_H(y)
    =C\frac{1_J(x)1_J(y)}{\mu(J)},
\end{split}
\end{equation*}
which is the correct bound. Since $D_J^{(k-j+r)}$ is a difference of averaging operators (or from the bounds for Haar functions), we find that the kernel of $D_J\sum_H A_H^{(k,2)}D_J^{(k-j+r)}$ has the same bound. This completes the proof.
\end{proof}

Now we have completed the proof of all claims in Theorem~\ref{thm:newRepr}, and therefore also the proof of Corollary~\ref{cor:A2}.

\section{The weak-type bounds}\label{sec:weak}

In this final section, we complete the proof of Theorem~\ref{thm:details} by proving the asserted weak-type estimates in the case that $\mu$ is a doubling measure.

\begin{proposition}
Let $\mu$ be a doubling measure on $\R^d$, and let $U_k$ be an operator of the form
\begin{equation*}
  U_k=\sum_{K\in\mathscr{D}}V_K^{(k)},\qquad
  V_K^{(k)}=1_K V_K^{(k)}P_K^{(k+r+1)}.
\end{equation*}
Suppose further that
\begin{equation*}
  \Norm{V_K^{(k)}}{L^1(\mu)\to L^1(\mu)}\leq C,\qquad\Norm{U_k}{L^2(\mu)\to L^2(\mu)}\leq C\sqrt{k}.
\end{equation*}
Then
\begin{equation*}
  \Norm{U_k}{L^1(\mu)\to L^{1,\infty}(\mu)}\leq Ck.
\end{equation*}
\end{proposition}

Observe that each of the operators $Q_k,R_k,Q_k^*,R_k^*$ is of the form $U_{\abs{k}}$ considered in the proposition.

\begin{proof}
Let $f\in L^1(\mu)$ and $\lambda>0$. We make the usual Calder\'on--Zygmund decomposition: Consider the maximal cubes $J\in\mathscr{D}$ such that $\ave{\abs{f}}_J>\lambda$; then by the doubling property $\ave{\abs{f}}_J\leq C\lambda$. Thus the good part
\begin{equation*}
  g:=\sum_J\ave{f}_J 1_J+f 1_{\Omega^c},\qquad\Omega:=\bigcup J
\end{equation*}
satisfies $\Norm{g}{\infty}\leq C\lambda$, $\Norm{g}{1}\leq\Norm{f}{1}$, and hence $\Norm{g}{2}^2\leq C\lambda\Norm{f}{1}$. It follows that
\begin{equation*}
  \mu(\{\abs{U_k g}>\lambda\})
  \leq \lambda^{-2}\Norm{U_k g}{2}^2
  \leq C\lambda^{-2} (\sqrt{k}\Norm{g}{2})^2
  \leq Ck\lambda^{-1}\Norm{f}{1}.
\end{equation*}
Also, the set $\Omega:=\bigcup J$ satisfies $\mu(\Omega)\leq \lambda^{-1}\Norm{f}{1}$. So it remains to estimate
\begin{equation*}
  \mu(\{\abs{U_k b}>\lambda\}\setminus\Omega),\qquad
  b:=\sum_J b_J:=\sum_J (f-\ave{f}_J)1_J.
\end{equation*}
We have
\begin{equation*}
  \mu(\{\abs{U_k b}>\lambda\}\setminus\Omega)
  \leq\frac{1}{\lambda}\int_{\Omega^c}\abs{U_k b}d\mu
  \leq\frac{1}{\lambda}\sum_J\int_{J^c}\abs{U_k b_J}d\mu,
\end{equation*}
and
\begin{equation*}
  1_{J^c}U_k b_J
  =1_{J^c}\sum_K V_K^{(k)}b_J
  =1_{J^c}\sum_{K\supsetneq J}V_K^{(k)}b_J,
\end{equation*}
since $K$ must intersect both $J$ and $J^c$ to get a non-zero contribution. Moreover, noting that
\begin{equation*}
  V_K^{(k)}b_J
  =V_K^{(k)}P_K^{(k+r+1)}b_J
  =V_K^{(k)}\Big(\sum_{I:I^{(k+r+1)}=K}1_I\ave{b_J}_I-1_K\ave{b_J}_K\Big),
\end{equation*}
and recalling that $\int_J b_Jd\mu=0$, we find that $\ave{b_J}_I=0$ for all dyadic cubes $I$ with $\ell(I)\geq\ell(J)$. So the only non-zero contribution can arise if $\ell(I)=2^{-k-r-1}\ell(K)<\ell(J)$. In combination with $J\subsetneq K$, this implies that $J\subsetneq K\subseteq J^{(k+r)}$. So altogether
\begin{equation*}
\begin{split}
  \int_{J^c}\abs{U_k b_J}d\mu
  &\leq\sum_{K:J\subsetneq K\subseteq J^{(k+r)}}\Norm{V_K^{(k)}b_J}{1} \\
  &\leq\sum_{K:J\subsetneq K\subseteq J^{(k+r)}} C\Norm{b_J}{1}
  =(k+r)C\Norm{b_J}{1}\leq Ck\Norm{b_J}{1},
\end{split}
\end{equation*}
since $k\geq 1$ and $r$ is fixed. Summing over $J$ we conclude that
\begin{equation*}
  \mu(\{\abs{U_k b}>\lambda\}\setminus\Omega)
  \leq\frac{1}{\lambda}\sum_J\int_{J^c}\abs{U_k b_J}d\mu
  \leq\frac{1}{\lambda}\sum_J Ck\Norm{b_J}{1}
  \leq\frac{Ck}{\lambda}\Norm{f}{1},
\end{equation*}
and this completes the proof.
\end{proof}


\end{document}